\documentclass[10pt,reqno]{amsart}
\usepackage{bbm}
\usepackage{mathrsfs}
\usepackage{diagbox}
\usepackage{cite}
\usepackage{amsfonts} %%% i.e. use 12pt type
\usepackage[dvipsnames,usenames]{color}
\usepackage{graphicx,latexsym,bm,amsmath,amssymb,verbatim,multicol,lscape}
\newtheorem{thm}{Theorem} [section]

\newtheorem{lem}[thm]{Lemma}

\theoremstyle{definition}

\theoremstyle{remark}

\numberwithin{equation}{section}
\newcommand{\F}{\mathbb{F}}
\begin{document}
\title{On the number of zeros of diagonal cubic forms over finite fields}
\begin{abstract}
Let $\F_q$ be the finite field with $q=p^k$ elements with
$p$ being a prime and $k$ be a positive integer. For any $y, z\in\mathbb{F}_q$,
let $N_s(z)$ and $T_s(y)$ denote the numbers of zeros of $x_1^{3}+\cdots+x_s^3=z$
and $x_1^3+\cdots+x_{s-1}^3+yx_s^3=0$, respectively.
Gauss proved that if $q=p, p\equiv1\pmod3$ and $y$ is non-cubic, then
$T_3(y)=p^2+\frac{1}{2}(p-1)(-c+9d)$, where $c$ and $d$ are uniquely determined by
$4p=c^2+27d^2,~c\equiv 1 \pmod 3$ except for the sign of $d$. In 1978, Chowla,
Cowles and Cowles determined the sign of $d$ for the case of $2$ being a non-cubic
element of $\F_p$. But the sign problem is kept open for the remaining case of $2$ being cubic in $\F_p$. In this paper, we solve this sign problem by determining
the sign of $d$ when $2$ is cubic in $\F_p$. Furthermore, we show that the generating
functions $\sum_{s=1}^{\infty} N_{s}(z) x^{s}$ and $\sum_{s=1}^{\infty} T_{s}(y)x^{s}$
are rational functions for any $z, y\in\mathbb F_q^*:=\mathbb F_q\setminus \{0\}$ with
$y$ being non-cubic over $\F_q$ and also give their explicit expressions.
This extends the theorem of Myerson and that of Chowla, Cowles and Cowles.
\end{abstract}

\author[S.F. Hong]{Shaofang Hong$^*$}
%    Address of record for the research reported here
\address{Mathematical College, Sichuan University, Chengdu 610064, P.R. China}
%\curraddr{}
\email{sfhong@scu.edu.cn; s-f.hong@tom.com; hongsf02@yahoo.com}
\author[C.X. Zhu]{Chaoxi Zhu}
\address{Mathematical College, Sichuan University, Chengdu 610064, P.R. China}
\email{zhuxi0824@126.com}
\thanks{$^*$S.F. Hong was supported partially by National 
Science Foundation of China Grant \#11771304.}
\keywords{Exponential sum, Gauss sum, Jacobi sum, generating function,
diagonal cubic form, Hasse-Davenport relation.}
\subjclass[2000]{Primary 11T23,11T24} 
\maketitle

\section{Introduction}
Let $p$ be a prime number and let $k$ and $s$ be positive integers. Denote
by $\mathbb{F}_q$ the finite field of $q=p^k$ elements and let
$f(x_1, ..., x_s)\in \mathbb{F}_q[x_1, ..., x_s]$ be a polynomial in $s\ge 2$
indeterminates. Counting the number $N(f=0)$ of zeros $(x_1, ..., x_s)\in \mathbb{F}_q^s$
of the equation $f(x_1, ..., x_s)=0$ is an important and fundamental topic
in number theory and finite field. An explicit formula for $N(f=0)$ is known
when $\deg(f)\le 2$ (see \cite{[IR]} or \cite{[LN]}). The $p$-adic behavior
of $N(f=0)$ has been deeply studied by many authors (see, for example,
\cite{[AS]}, \cite{[Ax]}, \cite{[Che]}, \cite{[K]}, \cite{[MM]},
\cite{[Wan2]} and \cite{[War]}).
Finding the number of solutions $(x_1, ... , x_s)\in \mathbb{F}_q^s$
of the general diagonal equation $a_1x_1^{d_1}+...+a_sx_s^{d_s}=b \
(s\ge2, a_1, ..., a_s\in\mathbb{F}_q^*:=\mathbb F_q\setminus \{0\}, b\in\mathbb{F}_q)$
over $\mathbb{F}_q$ is a difficult problem. The special case where all
the $d_i$ are equal has extensively been studied (see, for example,
\cite{[CL]}, \cite{[J]}, \cite{[Mor]}, \cite{[Wan1]} and \cite{[WJ1]}).
This is the example chosen by Weil \cite{[We]} to illustrate his
renowned conjecture on projective varieties over finite
fields, which was proved later by Deligne \cite{[De]}. Even in this
special situation, only a few special cases are known where the number
of solutions can be explicitly calculated.

For an element $z\in\mathbb F_q$, one lets $N_s(z)$ denote the number of zeros
$(x_1, ..., x_s)\in\mathbb{F}_q^s$ of the following diagonal cubic
form over $\mathbb F_q$:
\begin{equation*}\label{4}
x_1^3+\cdots+x_s^3=z.
\end{equation*}
Chowla, Cowles and Cowles \cite{[CCC77]} \cite{[CCC78]} initiated
the investigation of $N_s(0)$. When $k=1$, i.e. $q=p$, it is easy
to see that $N_s(0)=p^{s-1}$ if $p\equiv2\pmod3$. However, for the
case $p\equiv1\pmod3$, the situation becomes complicated.
Chowla, Cowles and Cowles \cite{[CCC77]} proved that
the generating function $\sum_{s=1}^{\infty} N_s(0) x^{s}$
is a rational function of $x$. Actually, they showed that the following is true:
$$
\sum_{s=1}^{\infty} N_s(0) x^{s}=\frac{x}{1-p x}+\frac{(p-1)(2+cx)x^{2}}{1-3 p x^{2}-p c x^{3}},
$$
where $c$ is uniquely determined by $4p=c^{2}+27d^{2}~~\text {and}~~c\equiv1\pmod3.$
From this, one can read an expression of $N_s(0)$ for each integer $s\ge 1$.
Two years later, Myerson extended the Chowla-Cowles-Cowles theorem
by proving the following result.
\begin{thm}\label{MMMM} \textup{\cite{[Myerson]}}
Let $q\equiv1\pmod3$. Then
$$
\sum_{s=1}^{\infty} N_s(0) x^{s}=\frac{x}{1-q x}+\frac{(q-1)(2+cx)x^{2}}{1-3 q x^{2}-q c x^{3}},
$$
where $c$ is uniquely determined by
\begin{equation}\label{4'}
4q=c^2+27d^2,~c\equiv1~(\bmod~3)~\text {\it and~if~}~~ p\equiv1~(\bmod~3), \ {\it then}~~(c,p)=1.
\end{equation}
\end{thm}  

Let $z\in\mathbb F_q^{\ast}:=\mathbb F_q\setminus \{0\}$. If $q\equiv2\pmod3$,
it is known that every element in $\mathbb F_q$ is a cube, and so $N_s(z)=q^{s-1}$.
If $q\equiv1\pmod3$ with $p\equiv2\pmod3$, then Wolfmann \cite{[WJ2]} gave a
formula for $N_s(z)$ and did not obtain the explicit expression
for $\sum_{s=1}^{\infty}N_s(z)x^s$. Furthermore, the formula for $N_s(z)$ is
still unknown when $q\equiv 1\pmod 3$ with $p\equiv 1\pmod 3$. On the other
hand, Richman  \cite{[Richman]} used the Cayley-Hamilton theorem to show
that the sequence $\{N_s(z)\}_{s=1}^\infty$ satisfies a linear recurrence
relation.

In this paper, by using the Gauss
sum, Jacobi sum and Hasse-Davenport relation, we show that the generating
function $\sum_{s=1}^{\infty}N_s(z)x^s$ is a rational function and also
present its explicit expression. To state this result, we let $g$ be a
fixed generator of $\mathbb F_q^{\ast}$, i.e. $\mathbb F_q^{\ast}=\langle g\rangle$.
Then $z=g^e$ for some integer $e$ with $1\le e\le q-1$,
and $e$ is called the {\it index} of $z$ respective
to $g$ and denoted by $e={\rm ind}_g(z)$. For any integer $n$, let
$\langle n\rangle_3$ stand for the unique integer $i$ with $0\le i\le 2$
such that $n\equiv i\pmod 3$. For any $x \in\mathbb R\setminus\{0\}$ and
$a\in \F_q$, let ${\rm sgn}(x):=\frac{|x|}{x}$ be the {\it sign} of $x$
and ${\rm N}_{\F_{q}/\F_{p}}(a):=\sum_{i=0}^{s-1}a^{p^i}$ be
the {\it Norm map} of $a$ from $\mathbb F_q$ to $\mathbb F_p$.
We can now state our main result.

\begin{thm}\label{thm3}
Let $z\in\mathbb F_q^{\ast}=\langle g\rangle$ and $q=p^k\equiv1\pmod3$ with $k$
being a positive integer. Then
$$\sum_{s=1}^{\infty}N_s(z)x^s =\frac{x}{1-qx}
+\frac{2x+(c-2)x^2-cx^3}{ 1-3qx^2-qcx^3}$$
if $z$ is cubic, where $c$ is
uniquely determined by (1.1), and
$$\sum_{s=1}^{\infty}N_s(z)x^s =\frac{x}{1-qx}
-\frac{x+\big(2+\frac{c}{2}+\frac{9d}{2}\delta_{z}(q)\big)x^2+cx^3}{ 1-3qx^2-qcx^3}$$
if $z$ is non-cubic, where $c$ and $d$ are
uniquely determined by (\ref{4'}) with $d\ge 0$ and
\begin{equation}\label{5}
\delta_{z}(q)=\left\{\begin{array}{lll}{(-1)^{\langle{\rm ind}_g(z)\rangle_3}\cdot{\rm sgn}
\big({\rm Im}\big((r_1+3\sqrt3r_2{\rm i})^k\big)\big)} & {\text {\it if }} & {k\equiv1\pmod2},  \\
{0} & {\text {\it if }} & {k\equiv0\pmod2,}\end{array}\right.
\end{equation}
where $r_1$ and $r_2$ are uniquely determined by:
$$4p=r_1^2+27 r_2^2, \ r_1\equiv1\pmod3, \
9r_2\equiv\big(2{\rm N}_{\F_{q}/\F_{p}}(g)^{\frac{p-1}{3}}+1\big)r_1\pmod{p}.$$
\end{thm}

Now let $T_s(y)$ be the number of zeros of the following diagonal
cubic equation over $\mathbb F_q$:
$x_1^3+\cdots+x_{s-1}^3+yx_s^3=0$
with $y\in\mathbb F_q^{\ast}$ being non-cubic. In this paper,
we show that the generating function
$\sum_{s=1}^{\infty}T_{s+1}(y)x^s$ is rational. Actually,
with the help of Theorem 1.2, we can show the following result.
\begin{thm}\label{thmt1}
Let $y\in\mathbb{F}_q^*=\langle g\rangle$ be non-cubic and
$q=p^k\equiv1\pmod3$ with $k$ being a positive integer. Then
$$\sum_{s=1}^{\infty}T_{s+1}(y)x^s =\frac{qx}{1-qx}
-\frac{(q-1)x+(q-1)(\frac{c}{2}
+
\frac{9d}{2}\delta_{y}(q))x^2}{ 1-3qx^2-qcx^3},$$
where $c$ and $d$ are uniquely determined by (\ref{4'})
with $d\ge 0$ and $\delta_{y}(q)$ is given as in (1.2).
\end{thm}

As pointed out in {\cite{[CCC78]}}, the following
is essentially included in the derivation of
the cubic equation of periods by Gauss \cite{[Gauss]}:
Let $p\equiv1\pmod3$ and $y$ be non-cubic over $\mathbb F_p$.
Then the number $T_3(y)$ of zeros of
$x_1^3+x_{2}^3+yx_3^3=0$ over $\mathbb F_p$ is given by
$$T_3(y)=p^2+\frac{1}{2}(p-1)(-c+9d),$$
where $c$ and $d$ are uniquely determined by (except for the sign of $d$)
\begin{equation}\label{5}
4p=c^{2}+27d^{2} \ \text{\rm~~ and~~} \ c\equiv1~(\bmod~3).
\end{equation}
Chowla, Cowles and Cowles \cite{[CCC78]}
determined the sign of $d$ over $\mathbb{F}_p$
for the case of $2$ being non-cubic over $\mathbb F_p$
as the following result shows.

\begin{thm}\textup{\cite{[CCC78]}}
Let $p$ be a prime number such that $p\equiv1\pmod3$, $g^{\prime}$ be a
generator of $\F_p^{*}$ and $2$ be non-cubic over $\mathbb F_p$.
Let $y$ be non-cubic over $\mathbb F_p$ and $T_3(y)$ the number of
zeros of $x_1^3+x_{2}^3+yx_3^3=0$ over $\mathbb F_p$. Then
$$T_3(y)=p^2+\frac{1}{2}(p-1)(-c+9d),$$
where $c$ and $d$ are uniquely determined by (\ref{5}) with
$$d\equiv c\pmod4 \ \text{\it{if}~~} \
{\rm ind}_{g^{\prime}}(y)\equiv {\rm ind}_{g^{\prime}}(2) \pmod3 $$
and
$$d\not\equiv c\pmod4 \ \text{\it{if}~~} \
{\rm ind}_{g^{\prime}}(y)\equiv {\rm ind}_{g^{\prime}}(4) \pmod3.$$
\end{thm}

\noindent However, the sign of $d$ over $\mathbb{F}_p$ has not
been determined yet for the remaining case of $2$ being cubic
over $\mathbb F_p$. In this paper, we solve this Gauss sign problem.
In fact, we have the following more general result.

\begin{thm}\label{thmt2}
Let $y\in\mathbb F_q^{\ast}=\langle g\rangle$ be non-cubic
and $q=p^k\equiv1\pmod3$ with $k$ being a positive integer. Then
$$T_{3}(y)=q^2+\frac{1}{2}(q-1)(-c-9\delta_{y}(q)d),$$
where $c$ and $d$ are uniquely determined by (\ref{4'})
with $d\ge 0$ and $\delta_{y}(q)$ is given as in (1.2).
\end{thm}

Letting $q=p$ in Theorem 1.5 gives immediately the following result.

\begin{thm}\label{thmt3}
Let $p$ be a prime number such that $p\equiv1\pmod3$
and $y\in\mathbb F_p^{\ast}=\langle g\rangle$ be non-cubic.
Then
$$T_{3}(y)=p^2+\frac{1}{2}(p-1)(-c-9\delta_{y}(p)d),$$
where $c$ and $d$ are uniquely determined by (\ref{4'})
with $d\ge 0$ and
$$\delta_{y}(p)=(-1)^{\langle{\rm ind}_g(y)\rangle_3}\cdot{\rm sgn}
\big({\rm Im}(r_1+3\sqrt3r_2{\rm i})\big).$$
\end{thm}
\noindent Evidently, Theorem 1.6 extends the Chowla-Cowles-Cowles
result (Theorem 1.4) and answers completely the Gauss sign problem.

This paper is organized as follows. First of all, in Section 2,
we show several preliminary lemmas that are needed in the proof
of Theorem 1.2. Consequently, we supply in Section 3 the proof
of Theorem 1.2. In Section 4, we present the proofs of Theorems
\ref{thmt1} and \ref{thmt2}. An example is also given in
Section 4 to demonstrate the validity of Theorem 1.6.

\section{Auxiliary lemmas}
In this section, we first present some well-known results.

\begin{lem}{\rm\cite{[Myerson]}}\label{M}
Let $\mathbb F_q$ be the finite field with $q$ elements and
$\psi$ be a nontrivial additive character of $\mathbb F_{q}$.
Then for any element $x$ of $\mathbb F_q$, we have
$$
\sum_{a \in \mathbb{F}_{q}} \psi(ax)
=\left\{\begin{array}{lll}{q} & {\text {\it if }} & {x=0,} \\
{0} & {\text {\it if }} & {x \neq 0.}\end{array}\right.
$$
\end{lem}
For any multiplicative character $\chi$ of $\mathbb F_q$ and additive
character $\psi$ of $\mathbb F_q$, the
{\it Gauss sum} $G(\chi,\psi)$ is defined by
$$G(\chi,\psi):=\sum_{x\in \mathbb F_q^\ast}\chi(x)\psi(x).$$
One has the following result.
\begin{lem}{\rm \cite{[LN]}}\label{G}
Let $\chi$ be a nontrivial multiplicative and $\psi$ a nontrivial
additive character of $\mathbb F_q$. Then $|G(\chi,\psi)|=\sqrt{q}$
and $G(\chi,\psi)G(\overline{\chi},\psi)
=\chi(-1)q.$
\end{lem}

Recall that ${\rm N}_{\F_{q}/\F_{p}}(x):=\prod_{i=0}^{s-1}x^{p^{i}}$ is the
{\it Norm map} from $\mathbb F_{q}$ to $\mathbb F_p$ and
${\rm Tr}_{\F_{q}/\F_{p}}(x):=\sum_{i=0}^{s-1}x^{p^{i}}$ is the
{\it Trace map} from $\mathbb F_{q}$ to $\mathbb F_p$. The renowned
Hasse-Davenport relation can be stated as follows.

\begin{lem} {\rm (Hasse-Davenport)} {\rm \cite{[LN]}}\label{DH}
Let $\F_p$ be a finite field, $\F_{p^k}$ be the finite extension of $\F_p$
with $[\F_{p^k}: \F_p]=k$. Let $\chi^{'}$ be a multiplicative character,
$\psi^{'}$ an additive character of $\mathbb F_p$, not both of them trivial.
Let $\chi$ and $\psi$ be the lifts of $\chi^{'}$ and $\psi^{'}$ from $\F_p$
to $\F_{p^k}$, i.e. $\chi=\chi'\circ {\rm N}_{\F_{q}/\F_{p}}$ and
$\psi=\psi'\circ {\rm Tr}_{\F_{q}/\F_{p}}$. Then
$G(\chi, \psi)=(-1)^{k-1}G^k(\chi^{'},\psi^{'}).$
\end{lem}

The next lemma tells us when $p\equiv1\pmod3$, the multiplicative
character $\chi$ of order $3$ of $\F_q$ can be lifted by a
multiplicative character of order $3$ of $\F_p$.

\begin{lem}{\rm\cite{[LN]}}\label{TTSS}
Let $\F_p$ be a finite field and $\F_{p^k}$ be a
extension of $\F_p$. A multiplicative character $\chi$ of
$\mathbb F_q$ can be lifted by a multiplicative character $\chi^{'}$ of
$\F_p$ {\rm (``lift'' is defined as in Lemma \ref{DH})}
if and only if $\chi^{p-1}$ is trivial.
\end{lem}

Let $\lambda_1,\cdots,\lambda_n$ be nontrivial multiplicative
characters of $\F_q$. Then the sum
$$J(\lambda_1,\cdots,\lambda_n):=\sum_{(c_1,\cdots,c_n)\in \F_q^n
\atop c_1+\cdots+c_n=1}\lambda_1(c_1)\cdots\lambda_n(c_n)$$
is called a {\it Jacobi ~~sum} in $\F_q$. The following gives
a relation between Gauss sum and Jacobi sum.

\begin{lem}{\rm\cite{[LN]}}\label{JJJ}
Let $\lambda_1, ...,\lambda_{n-1}$ and $\lambda_n$ be nontrivial
multiplicative characters of $\F_q$ such that the product
$\lambda_1\cdots\lambda_n$ is nontrivial. Let $\psi$ be
a nontrivial additive character of $\F_q$. Then
$$J(\lambda_1,\cdots,\lambda_n)=\frac{G(\lambda_1,\psi)\cdots
G(\lambda_n,\psi)}{G(\lambda_1\cdots\lambda_n,\psi)}.$$
\end{lem}

\begin{lem}\label{JLB}{\rm \cite{[BEW]}}
Let $p\equiv1\pmod3$ and $\chi$ be a multiplicative character
of order $3$ over $\mathbb{F}_p$. Then
$$2J(\chi,\chi)=r_1+3\sqrt3r_2{\rm i},$$
where $r_1$ and $r_2$ are uniquely determined by:
$$4p=r_1^2+27r_2^2,~~r_1\equiv1\pmod3,
~~9r_2\equiv\big(2g^{\prime\frac{p-1}{3}}+1\big)r_1\pmod{p}.$$
with $g^{\prime}$ being the generator of $\mathbb{F}_p^*$ such that
$\chi(g^{\prime})=\frac{-1+\sqrt{3}{\rm i}}{2}$.
\end{lem}

\begin{lem}{\rm\cite{[M]}}\label{S}
Let $q\equiv1\pmod3$ and for any $h\in \mathbb F_q$, define
$$S_h:=\sum_{y\in \mathbb F_q}\psi(hy^3).$$
Then for any generator $g$ of $\mathbb F_q^{\ast}$,
$S_g, S_{g^2}$ and $S_{g^3}$ are the roots of the cubic equation
$x^3-3qx-qc=0$, where $c$ is uniquely determined by (\ref{4'}).
\end{lem}

From the definition of $S_h$, we can deduce that
\begin{eqnarray}\label{sg}
% \nonumber to remove numbering (before each equation)
  S_{g^{3k+i}}=S_{g^i}
\end{eqnarray}
for any integer $k$ and $i$. Consequently, we show several lemmas
that are needed in the proof of Theorem 1.2.

\begin{lem}\label{2.4}
Let $q\equiv1\pmod3$ and $\mathbb F_q^{\ast}=\langle g\rangle$. Then for
all $i\in\{0,1,2\}$ and any integer $m$, we have $N_s(g^{3m+i})=N_s(g^{i}).$
\end{lem}
\begin{proof}By the definition of $N_s(z)$, we have
\begin{align*}
N_{s}(g^{3m+i})&=\sum\limits_{(x_1,\cdots,x_s)\in \mathbb F_q^{s} \atop{ x_1^3+\cdots+x_s^3=g^{3m+i}}}1=\sum\limits_{(x_1,\cdots,x_s)\in \mathbb F_q^{s} \atop{(x_1/g^m)^3+\cdots+(x_s/g^m)^3=g^{i}}}1=\sum\limits_{(x_1,\cdots,x_s)\in \mathbb F_q^{s} \atop{ x_1^3+\cdots+x_s^3=g^{i}}}1=N_{s}(g^{i})
\end{align*}
as desired.
\end{proof}

In the rest of this paper, we let $\chi$ be the multiplicative character
of order $3$ which is defined for the generator $g$ of $\mathbb F_q^*$ by
$\chi(g):=\frac{-1+\sqrt3{\rm i}}{2}$
and $\psi$ be the canonical additive character
which is defined for any $x\in\mathbb F_q$ by
$$\psi(x):=e^{2\pi {\rm Tr}_{\F_{q}/\F_{p}}(x){\rm i}/p}.$$
We denote $\overline{\chi}$ the {\it conjugate} character of $\chi$,
which means that $\chi\overline{\chi}$ is the trivial multiplicative
character. For convenience, we let
$G:=G(\chi,\psi).$
By Lemmas \ref{G}, we have $G(\chi,\psi)
G(\overline{\chi},\psi)=\chi(-1)q=\chi^3(-1)q=q$
and $|G(\chi,\psi)|=|G(\overline{\chi},\psi)|=\sqrt{q}$, one can deduce
that $G(\overline{\chi},\psi)=\overline{G}$. We have the following result
about Gauss sums.
\begin{lem}\label{GS}
Let $q\equiv 1\pmod3$. Then
$G^3(\chi,\psi)+G^3(\overline{\chi},\psi)=cq,$
where $c$ is uniquely determined by (\ref{4'}).
\end{lem}
\begin{proof}
By Theorem \ref{MMMM}, we have
$$
\sum_{s=1}^{\infty} N_s(0) x^{s}=\frac{x}{1-q x}+\frac{x^{2}(q-1)(2+c x)}{1-3 q x^{2}-q c x^{3}},
$$
Comparing the coefficients of $x^2$ and $x^3$ on both sides gives us that
$N_2(0)=3q-2$ and $N_3(0)=q^2+cq-c$. Since
\begin{align*}
N_{3}(0)&=\sum\limits_{(x_1,x_2,x_3)\in \mathbb F_q^{3} \atop{ x_1^3+x_2^3+x_3^3=0}}1 \\
&=\sum_{(x_1,x_2)\in \mathbb F_q^2 \atop{ x_1^3+x_2^3=0}}1
    +\sum_{x_{3}\in \mathbb F_q^{\ast}}\sum_{(x_1,x_2)\in \mathbb F_q^2\atop {x_1^3+x_2^3=-x_{3}^3}} 1\\
&=N_2(0)+\sum_{x_{3}\in \mathbb F_q^{\ast}}\sum_{(x_1,x_2)\in \mathbb F_q^3\atop {x_1^3+x_2^3=-1}} 1 \\
&=N_2(0)+(q-1)N_{2}(-1),
\end{align*}
it follows that
\begin{eqnarray}\label{31}
N_2(-1)=\frac{N_3(0)-N_2(0)}{q-1}=q+c-2.
  \end{eqnarray}

On the other hand, by Lemma \ref{M}, we have
\begin{eqnarray*}
% \nonumber to remove numbering (before each equation)
  N_2(-1)=\frac{1}{q}\sum_{m\in \mathbb F_q}\sum_{(a_1,a_2)\in \mathbb F_q^2}
  \psi\big(m(a_1^3+a_2^3+1)\big).
  \end{eqnarray*}
Using the fact that $\psi(x+y)=\psi(x)\psi(y)$
for any $x, y \in \mathbb{F}_q$, we derive that
\begin{align}\label{eqN3}
     N_2(-1) &=\frac{1}{q}\sum_{m\in \mathbb F_q}\psi(m)
     \Big(\sum_{a\in \mathbb F_q}\psi(ma^3)\Big)^2\notag\\
     & =q+ \frac{1}{q}\sum_{m\in \mathbb F_q^{\ast}}\psi(m)
   \Big(\sum_{a\in \mathbb F_q}\psi(ma^3)\Big)^2.
\end{align}
Note that if $q\equiv1\pmod3$, then $x^3=h$ has exactly $3$ zeros over $\mathbb F_q$
if $h\in \mathbb F_q^*$ is cubic, and there is no zero of
$x^3=h$ if $h\in \mathbb F_q^*$ is not cubic. Therefore
\begin{align}\label{zzz}
% \nonumber to remove numbering (before each equation)
 \sum_{a\in \mathbb F_q} \psi(ma^3)
 =1+\sum\limits_{h\in \mathbb F_q^{\ast}\atop{h\text{\rm~ is a cubic}}}3\psi(mh).
\end{align}
Since $\chi$ is the multiplicative character of order $3$, we have
\begin{align}\label{aaa}
1+\chi(a)+\chi^2(a)=\left\{\begin{array}{lll}{0,} & {\text {\rm if }} & {a~{\rm  is \ non-cubic, }} \\
{3,} & {\text {\rm if }} & {a~{\rm is~ cubic. }}\end{array}\right.
 \end{align}
One puts (\ref{aaa}) into (\ref{zzz}) and obtains that
\begin{align}\label{eqN3'}
% \nonumber to remove numbering (before each equation)
 \sum_{a\in \mathbb F_q}\psi(ma^3)&=
 1+\sum_{a\in \mathbb F_q^{\ast}} (1+\chi(a)+\chi^2(a))\psi(ma)\notag\\
    &=\sum_{a\in \mathbb F_q}\psi(ma)+\sum_{a\in \mathbb F_q^{\ast}}
    \chi(a)\psi(ma)\notag+\sum_{a\in \mathbb F_q^{\ast}} \chi^2(a)\psi(ma)\notag\\
    &=\sum_{a\in \mathbb F_q^{\ast}}
\overline{\chi}(m)\chi(ma)\psi(ma)+\sum_{a\in \mathbb F_q^{\ast}} \chi(m)\overline{\chi}(ma)\psi(ma)\notag\\
    &= \overline{\chi}(m)G(\chi,\psi)+\chi(m)G(\overline{\chi},\psi):=\overline{\chi}(m)G+\chi(m)\overline{G}.
\end{align}
Putting (\ref{eqN3'}) into (\ref{eqN3}) and
applying Lemmas \ref{M} and \ref{G}, it follows that
\begin{align}\label{3121}
qN_2(-1)-q^2 &=\sum_{m\in \mathbb F_q^*}\Big(\overline{\chi}(m)G+\chi(m)\overline{G}\Big)^2
  \psi(m) \notag \\
  &=\sum_{m\in \mathbb F_q^*}\Big(\chi(m)G^2+2G\overline{G}+\overline{\chi}(m)\overline{G}^2\Big)\psi(m) \notag \\
  &=G^2\sum_{m\in \mathbb F_q^*}\psi(m)\chi(m)+2G\overline{G}\sum_{m\in \mathbb F_q^*}\psi(m)
  +\overline{G}^2\sum_{m\in \mathbb F_q^*}\overline{\chi}(m)\psi(m)\notag\\
  &= G^3-2G\overline{G}+\overline{G}^3
  =G^3+\overline{G}^3-2q.
\end{align}
Finally, comparing (\ref{31}) with (\ref{3121}), one arrives at
$G^3(\chi,\psi)+G^3(\overline{\chi},\psi)=cq$ as required.

This finishes the proof of Lemma \ref{GS}.
\end{proof}

Notice that $G(\chi,\psi)$ and $G(\overline{\chi},\psi)$ are conjugate
to each other and by Lemma \ref{GS}, one has
$G^3(\chi,\psi)+G^3(\overline{\chi},\psi)=cq$. Thus we can write
$G^3(\chi,\psi)=\frac{cq}{2}+v{\rm i}$ for some real number $v$.
But Lemma \ref{G} tells us $|G^3(\chi,\psi)|=q^{\frac{3}{2}}$, one derives that
$$q^3=G^3(\chi,\psi)\overline{G^3(\chi,\psi)}=\frac{c^2q^2}{4}+v^2.$$
It then follows from (1.1) that
$$v^2=\frac{4q^3-c^2q^2}{4}=\frac{q^2(4q-c^2)}{4}=\frac{27q^2d^2}{4}.$$
Therefore
\begin{eqnarray}\label{2.7}
% \nonumber to remove numbering (before each equation)
G^3(\chi,\psi)=\frac{cq}{2}+\theta(q)\frac{3\sqrt{3}qd}{2}{\rm i}
\end{eqnarray}
and
\begin{eqnarray}\label{2.7'}
% \nonumber to remove numbering (before each equation)
G^3(\bar\chi,\psi)=\frac{cq}{2}-\theta(q)\frac{3\sqrt{3}qd}{2}{\rm i}
\end{eqnarray}
with $c$ and $d\ge0$ being determined by (1.1) and
$$
\theta(q):=\left\{\begin{array}{lll}{1} & {\text {\rm if }} & {{\rm Im}(G^3(\chi,\psi))>0,} \\
{0} & {\text {\rm if }} & {{\rm Im}(G^3(\chi,\psi))=0,} \\
{-1} & {\text {\rm if }} & {{\rm Im}(G^3(\chi,\psi)) < 0}.\end{array}\right.
$$
The following result determines the value of $\theta(q)$.

\begin{lem}\label{llllll}
Let $q=p^k\equiv 1\pmod3$. With the notation above, we have
\begin{equation*}
\theta(q)=\left\{\begin{array}{lll}{{\rm sign}\big({\rm Im}\big(
(r_1+3\sqrt3r_2{\rm i})^k\big)\big)} & {\text {\it if }} & {k\equiv1\pmod2},  \\
{0} & {\text {\it if }} & {k\equiv0\pmod2,}\end{array}\right.
\end{equation*}
where $r_1$ and $r_2$ are uniquely determined by:
\begin{equation}\label{rr}
4p=r_1^2+27r_2^2,~~r_1\equiv1\pmod3,
~~9r_2\equiv\big(2{\rm N}_{\F_{q}/\F_{p}}(g)^{\frac{p-1}{3}}+1\big)r_1\pmod{p}.
\end{equation}
\end{lem}
\begin{proof}
In the case $k\equiv0\pmod2$, which means that $q$ is a square. By (\ref{4'}), we have $4q=c^2+27d^2$ which
implies $|c|=2p^{\frac{k}{2}}$ and $d=0$, it follows that $\theta(q)=0$.

In the case $k\equiv1\pmod2$, one can deduce $p\equiv1\pmod3$, it follows that $\chi^{p-1}$ is trivial.
By Lemma \ref{TTSS}, The cubic multiplicative character $\chi$ can be
lifted by a cubic multiplicative character $\chi^{\prime}$ of $\F_p$. One can check ${\rm N}_{\mathbb F_q/\mathbb F_p}(g)=g^{\frac{p^s-1}{p-1}}$
is a generator of $\mathbb F_p^{*}$ which satisfies
$\chi(g)=\chi^{\prime}\big({\rm N}_{\mathbb F_q/\mathbb F_p}(g)\big)=\frac{-1+\sqrt{3}{\rm i}}{2}$.
In Lemma \ref{JJJ}, letting $\lambda_1=\lambda_2=\chi^{\prime}$ and
$\psi^{\prime}$ be the canonical additive character over $\F_p$
and apply Lemma \ref{G} gives us that
$$J(\chi^{\prime},\chi^{\prime})=\frac{G^2(\chi^{\prime},
\psi^{\prime})}{G(\overline{\chi}^{'},\psi^{\prime})}
=\frac{G^3(\chi^{\prime},\psi^{\prime})}{G(\overline{\chi}^{'},
\psi^{\prime})G(\chi^{\prime},\psi^{\prime})}
=\frac{G^3(\chi^{\prime},\psi^{\prime})}{p}.$$

In Lemma \ref{JLB}, let $g^{\prime}={\rm N}_{\F_{q}/\F_{p}}(g)$. We have
$G^3(\chi^{\prime},\psi^{\prime})=pJ(\chi^{\prime},\chi^{\prime})
=\frac{p}{2}(r_1+3\sqrt3r_2{\rm i}),$
where $r_1$ and $r_2$ are uniquely determined by (\ref{rr}).
By the Davenport-Hasse relation (Lemma 2.3), we have
$G(\chi,\psi)=(-1)^{k-1}G^k(\chi^{\prime},\psi^{\prime}).$
Therefore
\begin{equation}\label{FH}
G^3(\chi,\psi)=(-1)^{k-1}G^{3k}(\chi^{\prime},\psi^{\prime})
=\big(\frac{p}{2}\big)^k(r_1+3\sqrt3r_2{\rm i})^k.
\end{equation}
If $k\equiv1\pmod2$, then
$\theta(q)={\rm sign}\big({\rm Im}\big(G^3(\chi,\psi)\big)\big)
={\rm sign}\big({\rm Im}\big((r_1+3\sqrt3r_2{\rm i})^k\big) \big)$
as one expects. So Lemma 2.10 is proved.
\end{proof}

Finally, we show the following result as the conclusion of this section.

\begin{lem}\label{thm2}
Let $q\equiv1\pmod3$ and $u_s(z)=N_s(z)-q^{s-1}$. Then
for $z=g,g^2,g^3$ and $s=1,2,3$, the values of $u_s(z)$
are given in the following table:
\begin{table}[!htbp]
\centering
\begin{tabular}{|c|c|c|c|}
\hline
\diagbox{$z$}{$u_s(z)$}{$s$}&$1$&$2$&$3$\\ %添加斜线表头
\hline
 $g$  & $-1$ & $-2-\frac{c}{2}+\frac{9d}{2}\theta(q)$  &  $-3q-c$\\
\hline
$g^2$ & $-1$ & $-2-\frac{c}{2}-\frac{9d}{2}\theta(q)$  &  $-3q-c$\\
\hline
$g^3$ & $2$  & $-2+c$                                  &  $6q-c$\\
\hline
\end{tabular}
\end{table}
\end{lem}

\begin{proof}
By Lemma \ref{M}, we have
\begin{align} \label{3131}
% \nonumber to remove numbering (before each equation)
  N_s(z) &=      \frac{1}{q}\sum_{m\in \mathbb F_q}\sum_{(a_1,\cdots,a_s)\in \mathbb F_q^s}\psi\big(m(a_1^3+\cdots+a_s^3-z)\big)
      \notag \\
   &=q^{s-1}+ \frac{1}{q}\sum_{m\in \mathbb F_q^{\ast}}\psi(-mz) \Big(\sum_{a\in \mathbb F_q}\psi(ma^3)\Big)^s,
\end{align}
Putting (\ref{eqN3'}) into (\ref{3131}) gives us that
\begin{eqnarray*}
   qu_s(z)=qN_s(z)-q^s = \sum_{m\in \mathbb F_q^{\ast}}\psi(-mz)\big(\overline{\chi}(m)G+\chi(m)\overline{G}\big)^s.
\end{eqnarray*}
It then follows from Lemmas \ref{M}, \ref{G} and \ref{GS} that
\begin{align}\label{2.9}
   qu_1(z) &=\sum_{m\in \mathbb F_q^{\ast}}\psi(-mz)\big(\overline{\chi}(m)G+\chi(m)\overline{G}\big) \notag \\
&=\big(\chi(-z)\overline{G}G+\overline{\chi}(-z)G\overline{G}\big) \notag\\
&=\chi(z)q+\overline{\chi}(z)q,
\end{align}
\begin{align}\label{2.10}
qu_2(z)&=\sum_{m\in \mathbb F_q^{\ast}}\psi(-mz)\big(\overline{\chi}(m)G+\chi(m)\overline{G}\big)^2\notag\\
&=\sum_{m\in \mathbb F_q^{\ast}}\psi(-mz)\big(\chi(m)G^2+\overline{\chi}(m)\overline{G}^2+2q\big)\notag\\
&= \overline{\chi}(-z)G^3+\chi(-z)\overline{G}^3-2q\notag\\
&=-2q+\overline{\chi}(z)G^3+\chi(z)\overline{G}^3
\end{align}
and
\begin{align}\label{2.11}
 qu_3(z)&=\sum_{m\in \mathbb F_q^{\ast}}\psi(-mz)\big(\overline{\chi}(m)G+\chi(m)\overline{G}\big)^3\notag\\
   &= \sum_{m\in \mathbb F_q^{\ast}}\psi(-mz)\big(G^3+\overline{G}^3
   +3\overline{\chi}(m)G^2\overline{G}+3\chi(m)G\overline{G}^2\big)\notag\\
   &=-cq+3\chi(-z)G^2\overline{G}^2+3\overline{\chi}(-z)G^2\overline{G}^2\notag\\
    &=-cq+3\big(\chi(z)+\overline{\chi}(z)\big)q^2.
\end{align}

Since
$$\chi(g)=\bar\chi(g^2)=\frac{-1+\sqrt3{\rm i}}{2},
\chi(g^2)=\bar\chi(g)=\frac{-1-\sqrt3{\rm i}}{2}$$
and
$$\chi(g^3)=\bar\chi(g^3)=1,$$
putting (\ref{2.7}) and (\ref{2.7'}) into (\ref{2.9}), (\ref{2.10})
and (\ref{2.11}), the desired result follows immediately.

The proof of Lemma \ref{thm2} is complete.
\end{proof}

\section{Proof of Theorem 1.2}

In this section, we present the proof of Theorem \ref{thm3}.\\

{\it Proof of Theorem 1.2.} First of all, by (\ref{3131}), we have
\begin{eqnarray*}
% \nonumber to remove numbering (before each equation)
  N_s(z)=q^{s-1}+ \frac{1}{q}\sum_{m\in \mathbb F_q^{\ast}}\psi(-mz)
  \Big(\sum_{x\in \mathbb F_q}\psi(mx^3)\Big)^s.
\end{eqnarray*}
It then follows from (2.1) that
\begin{align}\label{3.1}
&qu_s(z)\notag\\
= & qN_s(z)-q^s \notag \\
= &  \sum_{m\in \mathbb F_q^{\ast}}\psi(-mz)\Big(\sum_{x\in \mathbb F_q}\psi(mx^3)\Big)^s\notag\\
= &   \sum_{m \in \mathbb F_q^*\atop{{\rm ind}_g(m)\equiv1(\rm {mod}~3)}}\psi(-mz)S_m^s+
\sum_{m \in\mathbb F_q^*\atop{{\rm ind}_g(m)\equiv2(\rm {mod}~3)}}\psi(-mz)S_m^s+
\sum_{m \in \mathbb F_q^*\atop{{\rm ind}_g(m)\equiv0(\rm {mod}~3)}}\psi(-mz)S_m^s\notag\\
= &   S_g^s\sum_{m \in \mathbb F_q^*\atop{{\rm ind}_g(m)\equiv1(\rm {mod}~3)}}\psi(-mz)+
S_{g^2}^s \sum_{m \in \mathbb F_q^*\atop{{\rm ind}_g(m)\equiv2(\rm {mod}~3)}}\psi(-mz)+
S_{g^3}^s\sum_{m \in \mathbb F_q^*\atop{{\rm ind}_g(m)\equiv0(\rm {mod}~3)}}\psi(-mz)\notag\\
:=&S_g^s\eta_1(z)+S_{g^2}^s \eta_2(z)+ S_{g^3}^s\eta_3(z).
\end{align}
With $s$ replaced by $s-2$ and $s-3$ in (3.1), one obtains that
\begin{align}\label{3.2}
\sum_{i=1}^3S_{g^i}^{s-2}\eta_i(z)&=qu_{s-2}(z)
\end{align}
and
\begin{align}\label{3.3}
\sum_{i=1}^3S_{g^i}^{s-3}\eta_i(z)=qu_{s-3}(z).
\end{align}

On the other hand, for any integer $s\geq4$ and $i\in\{1,2,3\}$, one has by Lemma \ref{S}
$$S_{g^i}^s-3qS_{g^i}^{s-2}-qcS_{g^i}^{s-3}=0,$$
and so
$$S_{g^i}^s\eta_i(z)-3qS_{g^i}^{s-2}\eta_i(z)-qcS_{g^i}^{s-3}\eta_i(z)=0.$$
Then taking the sum gives us that
\begin{equation}\label{000'}
\sum_{i=1}^3S_{g^i}^s\eta_i(z)-3q\sum_{i=1}^3 S_{g^i}^{s-2}\eta_i(z)
-qc\sum_{i=1}^3 S_{g^i}^{s-3}\eta_i(z)=0.
\end{equation}
Hence putting the identities (\ref{3.1}) to (\ref{3.3})
into (\ref{000'}), one arrives at
$qu_s(z)-3q\big(qu_{s-2}(z)\big)-qc\big(qu_{s-3}(z)\big)=0$
which implies that
\begin{equation}\label{000}
u_s(z)-3qu_{s-2}(z)-qcu_{s-3}(z)=0
\end{equation}
for any $s\geq 4.$
It then follows from (\ref{000}) that
\begin{align*}
% \nonumber to remove numbering (before each equation)
  &(1-3qx^2-qcx^3)\sum_{s=1}^{\infty}u_s(z)x^s\\
=&\sum_{s=1}^{\infty}u_s(z)x^s-3q\sum_{s=1}^{\infty}u_s(z)x^{s+2}-qc\sum_{s=1}^{\infty}u_s(z)x^{s+3}\\
=&\sum_{s=1}^{\infty}u_s(z)x^s-3q\sum_{s=3}^{\infty}u_{s-2}(z)x^{s}-qc\sum_{s=4}^{\infty}u_{s-3}(z)x^{s}\\
=&u_1(z)x+u_2(z)x^2+\big(u_3(z)-3qu_1(z)\big)x^3 +\sum_{s=4}^{\infty}\big(u_s(z)-3qu_{s-2}(z)-qcu_{s-3}(z)\big)x^{s}\\
=&u_1(z)x+u_2(z)x^2+\big(u_3(z)-3qu_1(z)\big)x^3:=f_z(x).
\end{align*}
Therefore
\begin{align*}
% \nonumber to remove numbering (before each equation)
\sum_{s=1}^{\infty}u_s(z)x^s=\frac{f_z(x)}{ 1-3qx^2-qcx^3}.
\end{align*}
Since
$\sum_{s=1}^{\infty}q^{s-1}x^s=\frac{x}{1-qx}$
and $u_s(z)=N_s(z)-q^{s-1}$, we derive that
\begin{align}\label{4758495}
 \sum_{s=1}^{\infty}N_s(z)x^s =\sum_{s=1}^{\infty}\big(q^{s-1}+u_s(z)\big)x^s
 =\frac{x}{1-qx}+\frac{f_z(x)}{ 1-3qx^2-qcx^3}.
\end{align}

For any $z\in\mathbb{F}_q^*$, since $N_s(z)=N_s(g^{\langle{\rm ind}_g(z)\rangle_3})$
by Lemma \ref{2.4}, we need only to treat with the case $z\in\{g, g^2, g^3\}$.
By Lemma \ref{thm2}, we derive that
$$f_g(x)=-x-\big(2+\frac{c}{2}-\frac{9d}{2}\theta(q)\big)x^2-cx^3,
f_{g^2}(x)=-x-\big(2+\frac{c}{2}+\frac{9d}{2}\theta(q)\big)x^2-cx^3$$
and
$$f_{g^3}(x)=2x+(c-2)x^2-cx^3.$$
It then follows from (\ref{4758495}) that
$$\sum_{s=1}^{\infty}N_s(g^{3})x^s =\frac{x}{1-qx}+\frac{2x+(c-2)x^2-cx^3}{ 1-3qx^2-qcx^3},$$
$$
\sum_{s=1}^{\infty}N_s(g)x^s=\frac{x}{1-qx}-
\frac{x+\big(2+\frac{c}{2}-\frac{9d}{2}\theta(q)\big)x^2+cx^3}{ 1-3qx^2-qcx^3}
$$
and
$$\sum_{s=1}^{\infty}N_s(g^2)x^s =\frac{x}{1-qx}-
\frac{x+\big(2+\frac{c}{2}+\frac{9d}{2}\theta(q)\big)x^2+cx^3}{ 1-3qx^2-qcx^3}$$
as required.

Let $\delta_{z}(q):=(-1)^{\langle{\rm ind}_g(z)\rangle_3}\cdot\theta(q)$.
One can deduce that
$$\sum_{s=1}^{\infty}N_s(z)x^s =\frac{x}{1-qx}
-\frac{x+\big(2+\frac{c}{2}+\frac{9d}{2}\delta_{z}(q)\big)x^2+cx^3}{ 1-3qx^2-qcx^3}$$
if $z$ is not a cubic, and by Lemma \ref{llllll}, one derives that
\begin{equation*}
\delta_{z}(q)=\left\{\begin{array}{lll}{(-1)^{\langle{\rm ind}_g(z)\rangle_3}
\cdot{\rm sgn}\big(
{\rm Im}\big((r_1+3\sqrt3r_2{\rm i})^k\big)  \big)  } & {\text {\rm if }} & {k\equiv1\pmod2,} \\
{0} & {\text {\rm if }} & {k\equiv0\pmod2}.\end{array}\right.
\end{equation*}
This concludes the proof of Theorem \ref{thm3}. \hfill$\Box$

\section{Proofs of Theorems \ref{thmt1} and \ref{thmt2} and an example}
In this section, we prove Theorems \ref{thmt1} and \ref{thmt2}.
We begin with the proof of Theorem \ref{thmt1}.

\noindent{\it Proof of Theorem \ref{thmt1}.} Since $-1=(-1)^3$ is cubic, it follows that
\begin{align}\label{4.1}
T_s(y)&=\sum\limits_{(x_1,\cdots,x_{s})\in \mathbb F_q^{s} \atop{ x_1^3+\cdots+x_{s-1}^3+yx_{s}^3=0}}1 \notag\\
&=\sum_{(x_1,\cdots,x_{s-1})\in \mathbb F_q^{s-1} \atop{ x_1^3+\cdots+x_{s-1}^3=0}}1
+\sum_{x_{s}\in \mathbb F_q^{\ast}}\sum_{(x_1,\cdots,x_{s-1})
\in \mathbb F_q^{s-1}\atop {x_1^3+\cdots+x_{s-1}^3=-yx_{s}^3}} 1\notag\\
&=\sum_{(x_1,\cdots,x_{s-1})\in \mathbb F_q^{s-1} \atop{ x_1^3+\cdots+x_{s-1}^3=0}}1
+(q-1)\sum_{(x_1,\cdots,x_{s-1})\in \mathbb F_q^{s-1}\atop {x_1^3+\cdots+x_{s-1}^3=y}} 1 \notag \\
&=N_{s-1}(0)+(q-1)N_{s-1}(y).
\end{align}
By (\ref{4.1}), one can derive that
\begin{align}\label{4.2}
\sum_{s=1}^{\infty}T_{s+1}(y)x^s=(q-1)\sum_{s=1}^{\infty}N_{s}(y)x^s+\sum_{s=1}^{\infty}N_s(0)x^s.
\end{align}
Since $y\in \mathbb{F}^*_q$ is not cubic, applying Theorems 1.1 and 1.2
to $\sum_{s=1}^{\infty}N_{s}(y)x^s$ and $\sum_{s=1}^{\infty}N_s(0)x^s$,
respectively, by (4.2) we can derive that
\begin{align}
% \nonumber to remove numbering (before each equation)
 \sum_{s=1}^{\infty}T_{s+1}(y)x^s
=&(q-1)\Big(\frac{x}{1-qx}- \frac{x+\big(2+\frac{c}{2}
+\frac{9d}{2}\delta_{y}(q)\big)x^2+cx^3}{ 1-3qx^2-qcx^3}\Big)\notag\\
&~~~~~~+  \Big(\frac{x}{1-q x}+\frac{x^{2}(q-1)(2+c x)}{1-3 q x^{2}-q c x^{3}} \Big)  \notag \\
    =& \frac{qx}{1-qx}- \frac{(q-1)x+(q-1)\big(\frac{c}{2}
+\frac{9d}{2}\delta_{y}(q)\big)x^2}{ 1-3qx^2-qcx^3} \notag
\end{align}
as required. This finishes the proof of Theorem \ref{thmt1}.   \hfill$\Box$

Consequently, we present the proof of Theorem \ref{thmt2}. \\

\noindent{\it Proof of Theorem \ref{thmt2}.} By Theorem \ref{thmt1}, we have
\begin{align}
% \nonumber to remove numbering (before each equation)
(1-3qx^2-qcx^3)\sum_{s=1}^{\infty}(T_{s+1}(y)-q^s)x^s
=-(q-1)x-(q-1)\Big(\frac{c}{2}+\frac{9d}{2}\delta_{y}(q)\Big)x^2. \notag
\end{align}
Comparing the coefficients of $x^2$ on both sides gives us that
$T_{3}(y)=q^2+\frac{1}{2}(q-1)\big(-c-9d\delta_{y}(q)\big)$
as expected. This completes the proof of Theorem \ref{thmt2}. \hfill$\Box$\\

Finally, one observes that $p=31$ is the first prime number
satisfying that $p\equiv 1\pmod 3$ and $2$ is a cubic over
$\F_{p}$. In this case, Gauss \cite{[Gauss]} and Chowla,
Cowlers and Cowlers \cite{[CCC78]} did not determine the
exact value of the number $T_{3}(y)$ of zeros of
$x_1^3+x_{2}^3+yx_3^3=0$
over $\mathbb F_p$, where $y$ is non-cubic over $\mathbb F_p$.
We here can use Theorem 1.6 to give the exact value of $T_{3}(y)$
as the following example shows.\\

\noindent{\bf Example 4.1.} One can check that $g=3$ is a generator
of $\F_{31}^{*}$ and $2$ is a cubic element over $\F_{31}$.
If the positive integers $c, d, r_1$ and $r_2$ satisfy that
$4\cdot 31=c^2+27d^2 \text{~~with~~} c\equiv1\pmod3, ~~d\geq0$
and
$4\cdot31=r_1^2+27r_2^2,~~r_1\equiv1\pmod3,
~~9r_2\equiv\big(2\cdot3^{\frac{p-1}{3}}+1\big)r_1\pmod{31},$
then $c=4$, $d=2$, $r_1=4$ and $r_2=2$. Hence for any generator
$g$ of $\F_{31}^{*}$, one has $\delta_g(31)=(-1)\cdot{\rm sgn}
\big({\rm Im}(r_1+3\sqrt3r_2{\rm i})\big)=-1$ and $\delta_{g^2}(31)=1$.
It then follows from Theorem 1.6 that the numbers $T_3(g)$
and $T_3(g^2)$ of zeros $(x_1, x_2, x_3)\in{\mathbb F}_{31}^3$
of the cubic equations $x_1^3+x_{2}^3+gx_3^3=0$ and
$x_1^3+x_{2}^3+g^2 x_3^3=0$ over $\mathbb F_{31}$ are given by
$$T_{3}(g)=31^2+\frac{1}{2}(31-1)\big(-4-9\cdot2\cdot(-1)\big)=1171$$
and
$$T_{3}(g^2)=31^2+\frac{1}{2}(31-1)\big(-4-9\cdot2\big)=631,$$
respectively.   \hfill$\Box$

\bibliographystyle{amsplain}

\begin{thebibliography}{10}
\bibitem {[AS]} A. Adolphson and S. Sperber, $p$-Adic estimates for exponential
sums and the theorem of Chevalley-Warning, {\it Ann. Sci. 'Ecole Norm. Sup.}
{\bf 20} (1987), 545-556.

\bibitem{[Ax]} J. Ax, Zeros of polynomials over finite fields,
{\it Amer. J. Math.} {\bf 86} (1964), 255-261.

\bibitem{[BEW]} B. Berndt, R. Evans and K. Williams, {\it Gauss and Jacobi sums},
Wiley-Interscience, New York, 1998.

\bibitem{[Che]}  C. Chevalley, D${\rm\acute e}$monstration d${\rm\acute u}$ne hypoth${\rm\acute e}$se
de M. Artin (French), {\it Abh. Math. Sem. Univ. Hamburg} {\bf 11} (1935), 73-75.

\bibitem{[CCC77]}S. Chowla, J. Cowles and M. Cowles, On the number of
zeros of diagonal cubic forms, {\it J. Number Theory} {\bf9} (1977), 502-506.

\bibitem{[CCC78]}S. Chowla, J. Cowles and M. Cowles, The number of
zeros of $x^3+y^3+cz^3$ in certain finite fields, {\it J. Reine Angew. Math.}
{\bf 299 (300)} (1978), 406-410.

\bibitem{[CL]} L. Carlitz, The numbers of solutions of a particular equation
in a finite field, {\it Publ. Math. Debr.} {\bf 4} (1956), 379-383.

\bibitem{[De]}  P. Deligne, La conjecture de Weil II,
{\it Publ. Math. I.H.E.S.} {\bf 52} (1980), 137-252.

\bibitem{[Gauss]}C.F. Gauss, Disquisitiones arithmeticae,
Yale Univ. Press, New Haven, Conn., 1966.

\bibitem{[IR]} K. Ireland and M. Rosen, {\it A classical introduction to
modern number theory}, Second Edition, Springer-Verlag New York, Inc. 1990.

\bibitem{[J]} J.R. Joly, ${\rm\acute e}$quations et vari${\rm\acute e}$t${\rm\acute e}$s
alg${\rm\acute e}$briques sur un corps fini, {\it Enseign. Math.} {\bf 19} (1973), 1-117.

\bibitem{[K]} N.M. Katz, On a theorem of Ax, {\it Amer. J. Math.}
{\bf 93} (1971), 485-499.

\bibitem{[LN]}R. Lidl and H. Niederreiter, {\it Finite fields},
Second edition, Encyclopedia of Mathematics and its Applications,
Vol. 20, Cambridge University Press, Cambridge, 1997.

\bibitem{[MM]}O. Moreno and C.J. Moreno, Improvement of Chevalley-Warning
and the Ax-Katz theorem, {\it Amer. J. Math.} {\bf 117} (1995), 241-244.

\bibitem{[Mor]}B. Morlaye, ${\rm\acute e}$quations diagonales non
homog${\rm\acute e}$nes sur un corps fini,
{\it C. R. Acad. Sci. Paris Ser. A} {\bf 272} (1971), 1545-1548.

\bibitem{[Myerson]}G. Myerson, On the number of zeros of diagonal cubic forms,
{\it J. Number Theory} {\bf 11} (1979), 95-99.

\bibitem{[M]}G. Myerson, Period polynomials and Gauss sums for finite fields,
{\it Acta Arith.} {\bf 39} (1981), 251-264.

\bibitem{[Richman]} D.R. Richman, Some remarks on the number of solutions
to the equation $f(x_1)+...+f(x_n)=0$, {\it Stud. Appl. Math.} {\bf 71} (1984), 263-266.

\bibitem{[Wan1]} D. Wan, Zeros of diagonal equations over finite fields,
{\it Proc. Amer. Math. Soc.} {\bf 103} (1988), 1049-1052.

\bibitem{[Wan2]} D. Wan, An elementary proof of a theorem of Katz,
{\it Amer. J. Math.} {\bf 111} (1989), 1-8.

\bibitem{[War]} E. Warning, Bermerkung zur Vorstehenden Arbeit von Herr Chevalley,
{\it Abh. Math. Sem. Univ. Hamburg} {\bf 11} (1936), 76-83.

\bibitem{[We]} A. Weil, Numbers of solutions of equations in finite fields,
{\it Bull. Amer. Math. Soc.} {\bf 55} (1949), 497-508.

\bibitem{[WJ1]} J. Wolfmann, The number of solutions of certain diagonal
equations over finite fields, {\it J. Number Theory} {\bf 42} (1992), 247-257.

\bibitem{[WJ2]} J. Wolfmann, New results on diagonal equations over
finite fields from cyclic codes, {\it Contemp. Math.}
{\bf 168} (1994), 387-395.
\end{thebibliography}

\end{document}